\documentclass[11pt]{article}

\usepackage{amsmath,amssymb,graphicx,mathrsfs,hyperref,amsthm,color}

\usepackage{amsfonts}

\numberwithin{equation}{section}
\newtheorem{theorem}{Theorem}[section]

\newtheorem{lemma}[theorem]{Lemma}

\theoremstyle{definition}

\newtheorem{remark}[theorem]{Remark}
\newtheorem{example}[theorem]{Example}
\newtheorem{pf}{{\it Proof of Theorem 1.1}\!\!}

\newcommand{\Cinf}{\ensuremath{\mathcal{C}^\infty}}

\newcommand{\mb}[1]{\ensuremath{\mathbb{#1}}}
\newcommand{\N}{\mb{N}}

\newcommand{\R}{\mb{R}}

\newcommand{\G}{\ensuremath{{\mathcal G}}}

\newcommand{\EM}{\ensuremath{{\mathcal E}_{M}}}

\newcommand{\Neg}{\mathcal{N}}

\newcommand{\singsupp}{\operatorname{sing\,supp\hspace{0.5pt}}}
\newcommand{\supp}{\mathop{\mathrm{supp}}}

\newfont{\bigmath}{cmr12 at 13pt}

\newfont{\grecomath}{cmmi12 at 15pt}

\newcommand{\beq}{\begin{equation}}
\newcommand{\eeq}{\end{equation}}

\newcommand{\eps}{\varepsilon}

\newcommand{\cC}{{\mathcal C}}
\newcommand{\cD}{{\mathcal D}}

\newcommand{\cG}{{\mathcal G}}

\newcommand{\cV}{{\mathcal V}}

\begin{document}
\title{Propagation of singularities for generalized solutions to nonlinear wave equations}
\author
{Hideo Deguchi\\
   \texttt{hdegu@sci.u-toyama.ac.jp}\\
   Department of Mathematics, University of Toyama\\
   Gofuku 3190, 930-8555 Toyama, Japan
\and
Michael Oberguggenberger\\
  \texttt{michael.oberguggenberger@uibk.ac.at}\\
Arbeitsbereich f\"{u}r Technische Mathematik,
Universit\"{a}t Innsbruck \\
Technikerstra\ss e 13, A-6020 Innsbruck, Austria
}
\date{}
\maketitle

\begin{abstract}
The paper is devoted to regularity theory of generalized solutions to semilinear wave equations with a small nonlinearity. The setting is the one of Colombeau algebras of generalized functions. It is shown that in one space dimension, an initial singularity at the origin propagates along the characteristic lines emanating from the origin, as in the linear case. The proof relies on a fixed point theorem in the ultra-metric topology on the algebras involved. The paper takes up the initiating research of the 1970s on anomalous singularities in classical solutions to semilinear hyperbolic equations and transplants the methods into the Colombeau setting.
\end{abstract}



\section{Introduction}\label{sec : intro}

This paper addresses propagation of singularities for solutions to semilinear wave equations with a small nonlinearity. The equations are of the form
\begin{equation}\label{eq : nonlinear wave equation}
\begin{array}{l}
	\partial_t^2u - \partial_x^2 u = \varepsilon f(u),  \quad t \in [0,T], \ x \in \mathbb{R}, \vspace{4pt}\\
	u|_{t = 0} = u_{0},\quad \partial_tu|_{t = 0} = u_{1}, \quad x \in \mathbb{R}
\end{array}
\end{equation}
where $\varepsilon$ is small positive parameter and $f$ is smooth, polynomially bounded, and $f(0) = 0$. The initial data $u_0$ and $u_1$ are generalized functions of compact support, with a singularity at the origin. Approximating the initial data by nets of smooth functions
$(u_0^\varepsilon, u_1^\varepsilon)_{\varepsilon\in(0,1]}$, we establish the existence of a net of smooth solutions $(u^\varepsilon)_{\varepsilon\in(0,1]}$, up to an asymptotic error term of $O(\varepsilon^\infty)$. Measuring regularity in terms of estimates as $\varepsilon\downarrow 0$, we show that the initial singularity propagates only along the two characteristic lines emanating from the origin (the one-dimensional light cone), but the solution remains regular inside (and outside) the light cone.

The paper is formulated in the framework of Colombeau generalized functions. This will allow us to use the powerful tools from this theory to combine generalized function data with nonlinearities and to measure regularity.
Our main tool will be the Banach fixed point theorem in the so-called sharp topology, a complete ultra-metric topology on the Colombeau algebras. To our knowledge, this is the first time in the literature that such fixed-point arguments have been used to establish existence and regularity of solutions to nonlinear wave equations in the Colombeau framework.

Let us first put our result in perspective with regard to the classical literature. The discovery that in semilinear hyperbolic equations and systems, propagation of singularities does not necessarily occur along bicharacteristics emanating from singularities of the initial data, goes back to the paper \cite{Reed:1978} of Michael Reed, followed by the paper \cite{Rauch:1979} of Jeff Rauch. Michael Reed showed that for the nonlinear wave
equation\;\eqref{eq : nonlinear wave equation} (with $\varepsilon = 1$) in one space dimension, propagation of singularities is the same as in the linear case. That is, if classical initial data are smooth except at the origin, the solution will be smooth off the characteristic lines emanating from the origin. Jeff Rauch showed that this is not the case for space dimensions $d\geq 2$; there will be a loss of regularity inside the light cone. These two papers sparked a whole new research direction---the investigation of anomalous singularities in semilinear hyperbolic equations and systems, with a lot of activity until the 1990s (for a summary of that period, see \cite{Beals:1989}).

In order to describe the results of the paper, a few informal words about Colombeau algebras are required. Let $\Omega$ be an open subset of $\R^d$. Denote by $\mathcal{O}_M(\R)$ the space of smooth functions such that each derivative grows at most polynomially at infinity, and by $\mathcal{D}'(\Omega)$ the space of distributions on $\Omega$. Colombeau algebras are algebras of families $(u^\varepsilon)_{\varepsilon\in(0,1]}$ of smooth functions modulo asymptotically vanishing families, i.e., families all whose derivatives are of $O(\varepsilon^\infty)$ on compact sets as $\varepsilon\downarrow 0$.

A family $(u^\varepsilon)_{\varepsilon\in(0,1]}$ represents an element of the Colombeau algebra $\cG(\Omega)$, if every derivative
$\partial^\alpha u^\varepsilon$ is $O(\varepsilon^{b})$ on compact sets, for some $b\in\R$. The inclusion $\cD'(\Omega)\subset\cG(\Omega)$ holds, constructed by cut-off and convolution with a mollifier, and $\mathcal{C}^\infty(\Omega)$ is a faithful subalgebra. $\cG(\Omega)$ is invariant by superposition with maps $f\in\mathcal{O}_M(\R)$.

For the purpose of describing regularity properties of the elements of $\cG(\Omega)$, we shall single out two subalgebras. First, the subalgebra $\cG^\infty(\Omega)$ is characterized by the property that on every compact set, all derivatives $\partial^\alpha u^\varepsilon$ are $O(\varepsilon^{b})$ with the same $b$, depending only on the compact set. It holds that $\cG^\infty(\Omega) \cap \cD'(\Omega) = \cC^\infty(\Omega)$. However, $\cG^\infty(\Omega)$ is not invariant under nonlinear maps (except polynomials), hence does not deliver a general framework for regularity theory for nonlinear equations. Second, the subalgebra $\cG^0(\Omega)$ is characterized by the property that all derivatives $\partial^\alpha u^\varepsilon$ are $O(1)$ on every compact set. It also holds that
$\cG^0 \cap \cD'(\Omega) = \cC^\infty(\Omega)$. In addition, $\cG^0(\Omega)$ is invariant under superposition with arbitrary functions $f\in\cC^\infty(\R)$.

Existence of solutions to semilinear wave equations in Colombeau algebras has been known for a long time. For a globally Lipschitz nonlinearity $f\in\mathcal{O}_M(\R)$, existence and uniqueness of a solution to \eqref{eq : nonlinear wave equation} in $\cG(\R^2)$ follows, for example, from \cite{O:1987}.
In the globally Lipschitz case, existence and uniqueness of solutions in  $\cG(\R^{d+1})$ has been shown in space dimensions $d= 1,2,3$, for example, in \cite{Colombeau:1985, MORusso:1998}. For power nonlinearities, existence and uniqueness results in an $L^2$-based Colombeau algebra have been obtained in space dimensions $d\leq 9$ with suitable bounds on the polynomial growth of $f$ in \cite{NOP:2005}.

Intrinsic regularity theory in Colombeau algebras (i.e., without recourse to distributional limits), started with the introduction of $\mathcal{G}^\infty(\Omega)$ in \cite{O:1992}. It turned out that in linear partial differential equations, most of the classical regularity theory could be replicated with $\mathcal{G}^\infty(\Omega)$ in place of $\mathcal{C}^\infty(\Omega)$: elliptic regularity, hypoellipticity, microlocal elliptic regularity, wave front sets, propagation along bicharacteristics, including the techniques of pseudodifferential operators and Fourier integral operators with Colombeau amplitudes and phase functions
\cite{Garetto:2006, GarettoGO:2005, GarettoHoer:2006, GarettoHO:2009, GarettoO:2014, HOP:2006, NPS:1998, O:2008}.
In addition, the $\mathcal{G}^\infty$-singular support of solutions to linear wave equations with discontinuous coefficients could be precisely located in one space dimension and for radially symmetric solutions in higher space dimensions \cite{D:2011, DHO:2013, DO:2016, HdeH:2001}. However, as $\cG^\infty(\Omega)$ is not invariant under non-polynomial smooth maps, it could not be used for nonlinear equations.

So far, only few and special results on propagation of singularities in Colombeau solutions to nonlinear hyperbolic equations and systems have been obtained \cite{DMO:2008, O:2004, O:2006, O:2006a}. The present paper suggests that progress can be made in nonlinear hyperbolic equations when $\mathcal{G}^0(\Omega)$ is used for measuring regularity. In fact, we follow the historical path and present here Michael Reed's result in the setting of Colombeau algebras.

The main result of the paper is that in the one-dimensional semilinear wave equation, propagation of $\mathcal{G}^0$-singularities is the same as in the linear case, that is, a $\mathcal{G}^0$-singularity at the origin does not affect the $\mathcal{G}^0$-regularity inside the one-dimensional light cone. At this stage, we can prove this result only for small nonlinearities. This is dictated by the method of proof, following Michael Reed's fixed point argument, but in the ultra-metric topology on Colombeau algebras. In this ultra-metric topology, a map is a contraction if and only if it lowers growth in $\varepsilon\downarrow 0$, whence the factor $\varepsilon$ in equation\;\eqref{eq : nonlinear wave equation}.

By the way, it is clear that anomalous $\mathcal{G}^0$-singularities occur for $(3\times 3)$-hyperbolic systems in the Colombeau setting as well. One just needs to adapt the example of \cite{RauchReed:1981} or \cite{Travers:1997}. In the spirit of the result of Jeff Rauch \cite{Rauch:1979}, we also present an example where singularities spread into the interior of the light cone in three space dimensions.

The plan of the paper is as follows. In Section \ref{sec : 2} we recall the required notions from Colombeau theory. Section \ref{sec : 3} contains the main result on propagation of singularities for problem \eqref{eq : nonlinear wave equation}. It concludes with remarks on possible extension and the mentioned example of anomalous propagation of singularities.

\section{Colombeau algebras}\label{sec : 2}

We will employ the {\it special Colombeau algebra} of generalized functions denoted by $\G^{s}$ in \cite{GKOS:2001} (called the {\it simplified Colombeau algebra} in \cite{B:1990}). However, here we will simply use the letter $\G$ instead.
This section serves to recall the definitions and properties required for our purpose. For more details, see e.~g. \cite{Colombeau:1984, Colombeau:1985, GKOS:2001, NPS:1998, O:1992}.

Given a non-empty open subset $\Omega$ of $\mathbb{R}^n$, the space of real valued, infinitely differentiable functions on $\Omega$ is denoted by $\Cinf(\Omega)$, while $\Cinf(\overline{\Omega})$ refers to the subspace of functions all whose derivatives have a continuous extension up to the closure of $\Omega$.

Let $\Cinf(\Omega)^{(0,1]}$ be the differential algebra of all maps from the interval $(0,1]$ into $\Cinf(\Omega)$. Thus each element of $\Cinf(\Omega)^{(0,1]}$ is a family $(u^\varepsilon)_{\varepsilon \in (0,1]}$ of real valued smooth functions on $\Omega$. The subalgebra $\EM(\Omega)$ is defined by the elements
$(u^\varepsilon)_{\varepsilon \in (0,1]}$ of $\Cinf(\Omega)^{(0,1]}$ with the property that, for all $K \Subset \Omega$ and $\alpha \in \mathbb{N}_0^n$, there exists $b\in\R$ such that
\[
	\sup_{x \in K} |\partial^{\alpha} u^\varepsilon(x)| = O(\varepsilon^{b}) \quad {\rm as}\ \varepsilon \downarrow 0.
\]
The ideal $\Neg(\Omega)$ is defined by all elements $(u^\varepsilon)_{\varepsilon \in (0,1]}$ of $\Cinf(\Omega)^{(0,1]}$ with the property that, for all $K \Subset \Omega$, $\alpha \in\mathbb{N}_0^n$ and $a \ge 0$,
\[
	\sup_{x \in K} |\partial^{\alpha} u^\varepsilon(x)| = O(\varepsilon^a) \quad {\rm as}\ \varepsilon \downarrow 0.
\]
The {\it algebra} $\G(\Omega)$ {\it of generalized functions} is defined as the factor space
\[
	\G(\Omega) = \EM(\Omega) / \Neg(\Omega).
\]
The Colombeau algebra $\G(\overline{\Omega})$ on the closure of $\Omega$ is constructed in a similar way: the compact subsets $K$ occurring in the definition are now compact subsets of $\overline{\Omega}$, i.e., may reach up to the boundary. Since $\EM(\overline{\Omega}) \subset \EM(\Omega)$ and $\Neg(\overline{\Omega}) \subset \Neg(\Omega)$, there is a canonical map $\G(\overline{\Omega}) \to \G(\Omega)$. However, this map is not injective, as follows from the fact that $\Neg(\Omega)\cap \EM(\overline{\Omega}) \ne \Neg(\overline{\Omega})$.

\emph{Restrictions to open subsets.}
Let $\omega$ be an open subset of $\Omega$ and $U\in\cG(\Omega)$. Then the restriction $U|\omega$, obtained by restriction of representatives, is a well defined element of $\cG(\omega)$. The \emph{support} of a generalized function $U\in\cG(\Omega)$, denoted by $\supp U$, is the complement of the largest open set $\omega\subset\Omega$ such that
$U|\omega = 0$. An analogous definition applies to elements $U\in\G(\overline{\Omega})$ with the sets $\omega$, on which $U$ vanishes, taken as open in $\overline{\Omega}$.
Similarly, the restriction of elements of $U\in\cG(\Omega)$ or $U\in\G(\overline{\Omega})$ to lower dimensional linear subspaces can be defined. In our case, this will give a meaning to the initial values of elements of $\cG([0,T]\times \R)$ at $t=0$.

\emph{The ring of generalized numbers.} We let $\EM$ be the space of nets $(r^\varepsilon)_{\varepsilon\in(0.1]}$ of real numbers such that $|r^\varepsilon| = O(\varepsilon^b)$ as
$\varepsilon \downarrow 0$ for some $b\in\R$. Similarly, $\Neg$ comprises those sequences which are $O(\varepsilon^a)$ as
$\varepsilon \downarrow 0$ for every $a\geq 0$. The factor space $\widetilde{\R} = \EM/\Neg$ is the \emph{Colombeau ring of generalized numbers}.

\emph{Generalized functions of bounded type.}
Let $\Omega$ be an open subset of $\R^n$ and $L$ a subset of $\Omega$. A generalized function $U$ from $\cG(\Omega)$ is called of
\emph{bounded type on $L$}, if it has a representative $(u^\varepsilon)_{\varepsilon \in (0,1]}$ such that
\[
   \sup_{x\in L}|u^\varepsilon(x)| = O(1) \quad {\rm as}\ \varepsilon \downarrow 0.
\]
The subalgebra $\cG^0(\Omega)$ comprises the generalized functions from $\cG(\Omega)$ all whose derivatives of any order are of bounded type on compact sets.
The $\cG^0$-singular support $\singsupp_{{\mathcal G^{0}}}$ is the complement of the largest open set $\omega \subset \Omega$ such that $U|\omega$ belongs to $\cG^0(\omega)$.
The bounded type property is defined similarly for generalized functions in $\cG(\overline{\Omega})$, as is the subalgebra $\cG^0(\overline{\Omega})$.

\emph{Superposition by smooth maps of polynomial growth.}
If $f\in \mathcal{O}_M(\mathbb{R})$ and $U\in \cG(\Omega)$, then $f(U)$ is a well defined element of $\cG(\Omega)$. That is, if $(u^\varepsilon)_{\varepsilon\in(0.1]}$ is a representative of $U$, then $(f(u^\varepsilon))_{\varepsilon\in(0.1]}$ belongs to $\EM(\Omega)$ and its class in $\cG(\Omega)$ does not depend on the choice of representative of $U$. In addition, $\cG^0(\Omega)$ is invariant under superposition by arbitrary smooth maps.

\emph{Integration.} Let $K$ be a compact subset of $\Omega$ and $U\in\cG(\Omega)$. Then $\int_K U(x) dx$, again defined on representatives, is a well defined element of $\widetilde{\mathbb{R}}$. A proof of this fact can be found in \cite{Colombeau:1985}. Similarly, if $\Gamma$ is a smooth, bounded curve in $\Omega$, then
$\int_\Gamma U d\xi$, where $d\xi$ denotes the line element, is a well defined element of $\widetilde{\mathbb{R}}$, see \cite{Aragona:2012}.

\emph{The sharp topology.} The sharp topology, introduced in \cite{B:1990}, can be defined through its family of neighborhoods $V(K,p,q)$, where $K$ is a compact subset of $\Omega$, $p \in \N$ and $q \geq 0$. An element $U$ of $\G(\Omega)$
belongs to $\cV(K,p,q)$ if it has a representative $(u^\varepsilon)_{\varepsilon \in (0,1]}$ such that
\begin{equation}\label{eq:nbhd}
   \sup_{x \in K} \max_{|\alpha|\leq p} |\partial^{\alpha} u^\varepsilon(x)| = O(\varepsilon^q) \quad {\rm as}\ \varepsilon \downarrow 0.
\end{equation}
We are going to recall that the sharp topology on $\cG(\Omega)$ can be defined in terms of an ultra-metric. This construction is due to \cite{Scarpalezos:1998, Scarpalezos:2000} and has been further developed by \cite{Garetto:2005}.
Let $(K_n)_{n\in\N}$ be an exhausting sequence of compact subsets of $\Omega$. The seminorms $\mu_n$ on $\cC^\infty(\Omega)$ are given by
\begin{equation}\label{eq : seminorms}
   \mu_n(f) = \sup_{x\in K_n} \sup_{|\alpha|\leq n}|\partial^\alpha f(x)|.
\end{equation}
Valuations $\nu_n : \EM(\Omega)\to(-\infty,\infty]$ can be defined by
\[
   \nu_n\big((u^\varepsilon)_{\varepsilon \in (0,1]}\big) = \sup_{b\in\R} \{\mu_n(u^\varepsilon) = O(\eps^b) \ {\rm as}\ \varepsilon \downarrow 0\}.
\]
Obviously, $(u^\varepsilon)_{\varepsilon \in (0,1]}$ belongs to $\Neg(\Omega)$ if and only if $\nu_n\big((u^\varepsilon)_{\varepsilon \in (0,1]}\big) = \infty$ for all $n$. Thus the valuations can be extended to the factor algebra
$\cG(\Omega)$. The following properties hold:
\begin{itemize}
\item[(a)] $\nu_n(U+V) \geq \min \big(\nu_n(U),\nu_n(V)\big)$;

\item[(b)] $\nu_n(UV) \geq \nu_n(U) + \nu_n(V)$;

\item[(c)] if $m\geq n$, then $\nu_m(U) \leq \nu_n(U)$.
\end{itemize}
By means of these evaluations, a family of ultra-pseudo-seminorms on $\cG(\Omega)$ can be defined by
\[
    p_n(U) = \exp(-\nu_n(U)).
\]
They have the properties
\begin{itemize}
\item[(a)] $p_n(U+V) \leq \max \big(p_n(U),p_n(V)\big)$;

\item[(b)]  $p_n(UV) \leq p_n(U)p_n(V)$;

\item[(c)] $p_n(\lambda U) = p_n(U)$ for all $\lambda \in \R$.
\end{itemize}
Finally, an ultra-metric can be defined on $\cG(\Omega)$ by
\[
   d(U,V) = \sum_{n=0}^\infty 2^{-n-1}\min\big(p_n(U-V),1\big).
\]
It is an easy exercise to show that the topology induced on $\cG(\Omega)$ by the ultra-metric $d$ is the same as the one given by the neighborhoods \eqref{eq:nbhd}.
Further, it is known \cite{ Garetto:2005, NPS:1998, Scarpalezos:1998} that $\cG(\Omega)$ with the uniform structure induced by $d$ is complete, hence a complete ultra-metric space.

On $\cG^0(\Omega)$, the ultra-metric simplifies to
\begin{equation}\label{eq:contrationG0}
   d(U,V) = \sum_{n=0}^\infty 2^{-n-1}p_n(U-V).
\end{equation}
Indeed, if $U\in \cG^0(\Omega)$, then $\nu_n(U) \geq 0$ for all $n$, hence $p_n(U)\leq 1$ for all $n$.
In particular, $\cG^0(\Omega)$ is contained in the unit ball around zero in $\cG(\Omega)$.

\section{Propagation of singularities}\label{sec : 3}

This section is devoted to presenting and proving our main result about propagation of singularities in the generalized solution to the
Cauchy problems \eqref{eq : nonlinear wave equation}.
We interpret problem $(\ref{eq : nonlinear wave equation})$ as
\begin{equation}\label{eq : generalized nonlinear wave equation}
\begin{array}{lr}
\partial_t^2U - \partial_x^2U = Ef(U) & \mbox{in}\ \G([0,T]\times\mathbb{R}),\vspace{4pt} \\
U|_{t = 0} = U_0,\quad \partial_tU|_{t = 0} = U_1 & \mbox{in}\ \G(\mathbb{R}) \vspace{4pt} \\
\end{array}
\end{equation}
in the Colombeau algebra of generalized functions. We make the following assumptions.

\begin{itemize}
\item[(A)]
$E$ is the generalized number with $(\varepsilon)_{\varepsilon \in (0,1]}$ as a representative;
\item[(B)] $f$ belongs to $\mathcal{O}_M(\mathbb{R})$ and satisfies $f(0) = 0$;
\item[(C)] $U_0$, $U_1$ belong to $\mathcal{G}(\mathbb{R}) \cap \mathcal{G}^0(\mathbb{R}\setminus \{0\})$ and are compactly supported. Furthermore, $U_0^{\prime}$, $U_1$ are of bounded type on $\R$.
\end{itemize}

\begin{theorem}\label{thm : propagation1}
Suppose that assumptions $(A),$ $(B)$ and $(C)$ hold. Then for any $T > 0$, there exists a solution $U \in \mathcal{G}([0,T] \times \mathbb{R})$ to problem $(\ref{eq : generalized nonlinear wave equation})$ such that
\begin{equation}\label{eq : singsupp of U}
\singsupp_{{\mathcal G^{0}}} U \subset \{(t,x) : |x|=t,\ 0 \le t \le T\}.
\end{equation}
\end{theorem}

Before proving Theorem \ref{thm : propagation1}, we first explain some notation and definitions that we will use.
Let $D_+$ be the directional derivative in the direction $\hat{t} + \hat{x}$, where $\hat{t}$ and $\hat{x}$ are unit vectors in the $t$ and $x$ directions, and let $D_-$ be the directional derivative in the direction $\hat{t} - \hat{x}$. Put
\begin{align*}
\Gamma_+ &= \{ (t,x) : x=t,\ 0 \le t \le T\}, \\
\Gamma_- &= \{ (t,x) : x=-t,\ 0 \le t \le T\}.
\end{align*}
By assumption (C), there exists $a > 0$ such that
\begin{equation}\label{eq : supp of initial data}
\supp U_0 \cup \supp U_1 \subset [-a,a].
\end{equation}
We denote by $\widetilde{\G}([0,T] \times \mathbb{R})$ the set of all elements $V$ of $\G([0,T] \times \mathbb{R})$ with the four properties that
\begin{itemize}
\item[(I)] for all $K \Subset [0,T] \times \mathbb{R}$, the function $V$ is of bounded type on $K$;
\item[(II)] for all $K \Subset [0,T] \times \mathbb{R} \setminus \Gamma_-$ and $\alpha \in \mathbb{N}$, the function $D_+^{\alpha}V$ is of bounded type on $K$;
\item[(III)] for all $K \Subset [0,T] \times \mathbb{R} \setminus \Gamma_+$ and $\alpha \in \mathbb{N}$, the function $D_-^{\alpha}V$ is of bounded type on $K$;
\item[(IV)] $V = 0$ on $\{(t,x) \in [0,T] \times \mathbb{R} : |x| > t + a\}$, where $a$ is the constant given in $(\ref{eq : supp of initial data})$.
\end{itemize}

We remark that properties (I) and (II) imply that $V$ is $\G^0$-regular on $[0,T] \times \mathbb{R} \setminus \Gamma_-$ in the $D_+$ direction. Properties (I) and (III) mean that $V$ is $\G^0$-regular on $[0,T] \times \mathbb{R} \setminus \Gamma_+$ in the $D_-$ direction.

We now introduce an ultra-metric $\widetilde{d}$ on $\widetilde{\G}([0,T] \times \R)$. Let $(K_n^{\pm})_{n\in\N}$ and $(K_n)_{n\in\N}$ be exhausting sequences of compact subsets of $[0,T] \times \mathbb{R} \setminus \Gamma_{\mp}$ and $[0,T] \times \mathbb{R}$, respectively. Take seminorms $\widetilde{\mu}_n$ on $\cC^\infty([0,T] \times \R)$ which are defined by
\begin{align*}
   \widetilde{\mu}_n(f) & = \sup_{(t,x)\in K_n^+} \sup_{1 \leq \alpha \leq n}|D_{+}^\alpha f(t,x)| + \sup_{(t,x)\in K_n^-} \sup_{1 \leq \alpha \leq n}|D_{-}^\alpha f(t,x)| \\
   & \quad + \sup_{(t,x) \in K_n}|f(t,x)|.
\end{align*}
Then valuations $\widetilde{\nu}_n$ and ultra-pseudo-seminorms $\widetilde{p}_n$ on $\widetilde{\G}([0,T] \times \R)$ can be defined in the same way as in Section \ref{sec : 2}.
Since $\widetilde{p}_n(V) \le 1$ for any $V \in \widetilde{\G}([0,T] \times \R)$, an ultra-metric $\widetilde{d}$ can be defined on $\widetilde{\cG}([0,T] \times \R)$ by
\[
   \widetilde{d}(V,W) = \sum_{n=0}^\infty 2^{-n-1} \widetilde{p}_n(V-W).
\]
The space $\widetilde{\G}([0,T] \times \R)$ equipped with the metric $\widetilde{d}$ is complete. This can be shown similarly to the proof of \cite[Theorem 1.1]{NPS:1998}.

We next prove the following two lemmas.

\begin{lemma}\label{lemma : 1}
Assume that $C_{\pm}(t,x) = \{(\tau,x\mp(t-\tau)) : 0 \le \tau \le t\}$.
Let $d\xi$ be the line element.
Then for any $A \in \widetilde{\G}([0,T] \times \mathbb{R}),$
\[
B_{\pm} = \int_{C_{\pm}(t,x)} A(\xi)\,d\xi \quad (\in \G([0,T] \times \mathbb{R}))
\]
belong to $\widetilde{\G}([0,T] \times \mathbb{R})$.
\end{lemma}

\begin{proof}
We will give the proof for $B_+$. The proof for $B_-$ is similar.

Let $(a^{\varepsilon})_{\varepsilon \in (0,1]}$ be a representative of $A$. Then $B_+$ has a representative
\[
b_{+}^{\varepsilon}(t,x) = \int_{C_{+}(t,x)} a^{\varepsilon}(\xi)\,d\xi.
\]
For any $K \Subset [0,T] \times \R$,
\[
\sup_{(t,x) \in K} |b_+^{\varepsilon}(t,x)| \le \sup_{(t,x) \in K} \sqrt{2}t \sup_{\xi \in C_+(t,x)} |a^{\varepsilon}(\xi)|.
\]
Hence there exists $K^{\prime} \Subset [0,T] \times \R$ $(K^{\prime} \supset K)$ such that
\[
\sup_{(t,x) \in K} |b_+^{\varepsilon}(t,x)| \le \sqrt{2}T \sup_{(t,x) \in K^{\prime}} |a^{\varepsilon}(t,x)|.
\]
This inequality and the assumption that $A$ satisfies property (I) yield that $B_+$ satisfies property (I).

A simple calculation shows that
\begin{align*}
D_+b_+^{\varepsilon}(t,x)
&= \lim_{h \downarrow 0} \dfrac{1}{h} \left[ \int_{C_+(t+h/\sqrt{2},x+h/\sqrt{2})} a^{\varepsilon}(\xi)\,d\xi - \int_{C_+(t,x)} a^{\varepsilon}(\xi)\,d\xi \right] \\
&= \lim_{h \downarrow 0} \dfrac{1}{h} \left[ \int_{(t,x)}^{(t+h/\sqrt{2},x+h/\sqrt{2})} a^{\varepsilon}(\xi)\,d\xi \right] \\
&= a^{\varepsilon}(t,x).
\end{align*}
Similarly, it holds that $D_+^{\alpha}b_+^{\varepsilon}=D_+^{\alpha-1}a^{\varepsilon}$ for any $\alpha \in \N$.
It follows from the assumption that $A$ satisfies property (II) that $B_+$ satisfies property (II).

The $D_-$-derivative of $b_+^{\varepsilon}$ is calculated as follows:
\begin{align*}
D_-b_+^{\varepsilon}(t,x)
&= \lim_{h \downarrow 0} \dfrac{1}{h} \left[ \int_{C_+(t+h/\sqrt{2},x-h/\sqrt{2})} a^{\varepsilon}(\xi)\,d\xi - \int_{C_+(t,x)} a^{\varepsilon}(\xi)\,d\xi \right] \\
&= \lim_{h \downarrow 0} \dfrac{1}{h} \left[ \int_{C_+(t,x)} (a^{\varepsilon}(\xi_1+h/\sqrt{2},\xi_2-h/\sqrt{2}) - a^{\varepsilon}(\xi_1,\xi_2))\,d\xi \right] \\
&\quad + \lim_{h \downarrow 0} \dfrac{1}{h} \int_{C_+(h/\sqrt{2},x-t-h/\sqrt{2})} a^{\varepsilon}(\xi)\,d\xi \\
&= \int_{C_+(t,x)} (D_-a^{\varepsilon})(\xi)\,d\xi + a^{\varepsilon}(0,x-t).
\end{align*}
Similarly for any $\alpha \in \N$
\begin{align*}
D_-^{\alpha} b_+^{\varepsilon}(t,x)
&= \int_{C_+(t,x)} (D_-^{\alpha}a^{\varepsilon})(\xi)\,d\xi + \sum_{j=1}^{\alpha} D_-^{\alpha-j}(D_-^{j-1}a^{\varepsilon})(0,x-t).
\end{align*}
Then for any $K \Subset [0,T] \times \R \setminus \Gamma_+$, there exists $K^{\prime} \Subset [0,T] \times \R \setminus \Gamma_+$ $(K^{\prime} \supset K)$ such that
\begin{align*}
\sup_{(t,x) \in K}|D_-^{\alpha} b_+^{\varepsilon}(t,x)|
& \le \sqrt{2}T \sup_{(t,x) \in K^{\prime}}|D_-^{\alpha}a^{\varepsilon}(t,x)| \\
& \quad + \sup_{(t,x) \in K}\sum_{j=1}^{\alpha} |D_-^{\alpha-j}(D_-^{j-1}a^{\varepsilon})(0,x-t)|.
\end{align*}
The first term on the right-hand side is uniformly bounded in $\varepsilon$, since $A$ satisfies property (III). If $(t,x) \in [0,T] \times \mathbb{R} \setminus \Gamma_+$, then $x-t \ne 0$, so for any $K \Subset [0,T] \times \mathbb{R} \setminus \Gamma_+$ and $\beta$, $\gamma$, $\delta \in \mathbb{N}_0$,
\[
\sup_{(t,x) \in K} |\partial_t^{\beta}\partial_x^{\gamma}(D_-^{\delta}a^{\varepsilon})(0,x-t)| = O(1) \quad {\rm as}\ \varepsilon \downarrow 0.
\]
This implies that for any $K \Subset [0,T] \times \mathbb{R} \setminus \Gamma_+$ and $\alpha$, $\delta \in \mathbb{N}_0$,
\[
\sup_{(t,x) \in K} |D_-^{\alpha}(D_-^{\delta}a^{\varepsilon})(0,x-t)| = O(1) \quad {\rm as}\ \varepsilon \downarrow 0.
\]
Therefore, $B_+$ satisfies property (III). It is immediate to check that $B_+$ satisfies property (IV). Thus $B_+$ belongs to $\widetilde{\G}([0,T] \times \R)$.
\end{proof}

\begin{lemma}\label{lemma : 2}
Assume that $V,$ $W$ belong to $\widetilde{\G}([0,T] \times \mathbb{R})$. If the integral
\begin{equation}\label{eq : integral}
\int_{-a-T}^{x} \dfrac{W -V}{2}\,dy \quad (\in \G([0,T] \times \mathbb{R}))
\end{equation}
vanishes on $\{(t,x) \in [0,T] \times \mathbb{R} : |x| > t + a\}),$ then it is in $\widetilde{\G}([0,T] \times \mathbb{R})$.
\end{lemma}

\begin{proof}
By assumption, the integral $(\ref{eq : integral})$ satisfies properties (IV). Since $V$, $W$ satisfy property (I), so does the integral $(\ref{eq : integral})$. Hence it remains to show that the integral $(\ref{eq : integral})$ satisfies properties (II) and (III).

By Lebesgue's dominated convergence theorem, we get
\begin{equation}\label{eq : formula1}
D_+ \int_{-a-T}^{x} \dfrac{W - V}{2}\,dy = \int_{-a-T}^{x} D_+ \dfrac{W - V}{2}\,dy.
\end{equation}
Similarly
\begin{equation}\label{eq : formula2}
D_- \int_{-a-T}^{x} \dfrac{W - V}{2}\,dy = \int_{-a-T}^{x} D_- \dfrac{W - V}{2}\,dy.
\end{equation}
On the other hand, noting that by assumption,
\[
\int_{-a-T}^{x} \dfrac{W - V}{2}\,dy = - \int_{x}^{a+T} \dfrac{W - V}{2}\,dy \quad \mbox{in} \ \G([0,T) \times \R),
\]
we have
\begin{align}
& D_+ \int_{-a-T}^{x} \dfrac{W - V}{2}\,dy = - \int_{x}^{a+T} D_+ \dfrac{W - V}{2}\,dy,
\label{eq : formula3}\\
& D_- \int_{-a-T}^{x} \dfrac{W - V}{2}\,dy = - \int_{x}^{a+T} D_- \dfrac{W - V}{2}\,dy.
\label{eq : formula4}
\end{align}
Any $K \Subset [0,T] \times \R \setminus \Gamma_-$ can be divided into two parts $K_1$, $K_2 \Subset [0,T] \times \R \setminus \Gamma_-$, where $K_1$ lies on the left of $\Gamma_-$ and $K_2$ is on the right of $\Gamma_-$. The boundedness on $K_1$ of $D_+$-derivatives of the integral $(\ref{eq : integral})$ follows from $(\ref{eq : formula1})$. The boundedness of $(\ref{eq : integral})$ on $K_2$ follows from $(\ref{eq : formula3})$. Therefore, $(\ref{eq : integral})$ satisfies property (II). Similarly, that $(\ref{eq : integral})$ satisfies property (III) follows from $(\ref{eq : formula2})$ and $(\ref{eq : formula4})$.
\end{proof}

We now turn to the proof of Theorem \ref{thm : propagation1}.

\begin{pf}

As mentioned in the introduction, we follow the ideas of Reed's article \cite{Reed:1978}, in which the phenomenon of propagation of singularities in classical solutions to problem $(\ref{eq : nonlinear wave equation})$ with $\varepsilon = 1$ has been studied.

We fix $T > 0$ arbitrarily and consider the Cauchy problem $(\ref{eq : generalized nonlinear wave equation})$. As stated above, by assumption (C), we have
\[
\supp U_0 \cup \supp U_1 \subset [-a,a]
\]
with some $a > 0$. Put $V= \partial_t U - \partial_xU$ and $W= \partial_tU + \partial_xU$. By assumption (B), we have $f(0) = 0$. Hence, by finite propagation speed, we can expect that $U$ vanishes on $\{(t,x) \in [0,T] \times \mathbb{R} : |x| > t + a\})$. If that is the case, then $U$ is expressed in the form
\[
U = \int_{-a-T}^{x} \dfrac{W-V}{2}\,dy,
\]
and so problem $(\ref{eq : generalized nonlinear wave equation})$ can be rewritten as the Cauchy problem for a first-order hyperbolic system
\begin{equation}\label{eq : system1}
\begin{array}{lr}
	(\partial_t + \partial_x)V = g_{V,W} & \mbox{in}\ \G([0,T]\times\mathbb{R}),\vspace{4pt} \\
	(\partial_t - \partial_x)W = g_{V,W} & \mbox{in}\ \G([0,T]\times\mathbb{R}),\vspace{4pt} \\
	V|_{t=0} = V_0 = U_1 - U_0^{\prime} & \mbox{in}\ \G(\mathbb{R}), \vspace{4pt} \\
      	W|_{t=0} = W_0 = U_1 + U_0^{\prime} & \mbox{in}\ \G(\mathbb{R}), \vspace{4pt} \\
\end{array}
\end{equation}
where
\[
g_{V,W} = E f\left(\int_{-a-T}^{x} \dfrac{W-V}{2}\,dy\right).
\]
We may rewrite problem $(\ref{eq : system1})$ as the pair of integral equations
\begin{equation}\label{eq : integral equation}
\begin{array}{l}
\displaystyle V(t,x) = V_0(x-t) + \int_0^t g_{V,W}(\tau,x-t+\tau)\,d\tau,\vspace{4pt} \\
\displaystyle W(t,x) = W_0(x+t) + \int_0^t g_{V,W}(\tau,x+t-\tau)\,d\tau. \vspace{4pt} \\
\end{array}
\end{equation}
Let
$\mathcal{M}([0,T] \times \R)$
denote the set
of all elements $(V,W)$ in $\widetilde{\G}([0,T] \times \R)^2$ which satisfy that
\begin{equation}\label{eq : property s}
\int_{-a-T}^{x} \dfrac{W - V}{2}\,dy = 0
\end{equation}
on $\{(t,x) \in [0,T] \times \mathbb{R} : |x| > t + a\})$.
Then $\mathcal{M}([0,T] \times \R)$ is a closed subspace of the complete metric space $\widetilde{\G}([0,T] \times \R)^2$
with the metric $d$ defined by
\[
d((V_1,W_1),(V_2,W_2)) = \widetilde{d}(V_1,V_2) + \widetilde{d}(W_1,W_2).
\]
This can be seen as follows. Take $(V,W)$ from $\widetilde{\G}([0,T] \times \R)^2 \setminus \mathcal{M}([0,T] \times \R)$. The function $(V,W)$ does not satisfy $(\ref{eq : property s})$ and so by \cite[Theorem 1.2.3]{GKOS:2001}, there are representatives $(v^{\varepsilon})_{\varepsilon \in (0,1]}$, $(w^{\varepsilon})_{\varepsilon \in (0,1]}$, a compact subset $K$ of $\{(t,x) \in [0,T] \times \mathbb{R} : |x| > t + a\})$ and a number $b \ge 0$ such that
\begin{equation}\label{eq : inequality 1}
\sup_{(t,x) \in K} \left|\int_{-a-T}^{x} \dfrac{w^{\varepsilon}(t,y) - v^{\varepsilon}(t,y)}{2}\,dy\right| \ge \varepsilon^b
\end{equation}
for sufficiently small $\varepsilon > 0$. We can choose $n \in \N$ and $C > 0$ large enough such that $K_n \supset K$ and further that for any $(z^{\varepsilon})_{\varepsilon \in (0,1]} \in \EM([0,T] \times \R)$
\begin{equation}\label{eq : inequality 2}
\sup_{(t,x) \in K} \left|\int_{-a-T}^{x} z^{\varepsilon}(t,y)\,dy\right|
\le C\sup_{(t,x) \in K_n} \left|z^{\varepsilon}(t,x)\right|
\end{equation}
for all $\varepsilon \in (0,1]$. Then the neighborhood $B((V,W);e^{-b-1}/2^{n+1})$ with center $(V,W)$ and radius $e^{-b-1}/2^{n+1}$ does not intersect $\mathcal{M}([0,T] \times \R)$. In fact, if $(\overline{V},\overline{W}) \in B((V,W);e^{-b-1}/2^{n+1})$, then $\widetilde{p}_n(\overline{V}-V) < e^{-b-1}$ and $\widetilde{p}_n(\overline{W}-W) < e^{-b-1}$.
These imply that $\widetilde{\nu}_n(\overline{V}-V) > b+1$ and $\widetilde{\nu}_n(\overline{W}-W) > b+1$, which in turn imply that $\widetilde{\mu}_n(\overline{v}^{\varepsilon}-v^{\varepsilon}) = O(\varepsilon^{b+1})$ and $\widetilde{\mu}_n(\overline{w}^{\varepsilon}-w^{\varepsilon}) = O(\varepsilon^{b+1})$, where $(\overline{v}^{\varepsilon})_{\varepsilon \in (0,1]}$, $(\overline{w}^{\varepsilon})_{\varepsilon \in (0,1]}$ are representatives of $\overline{V}$, $\overline{W}$, respectively. It follows from the definition of $\widetilde{\mu}_n$ that
\begin{align*}
   & \sup_{(t,x) \in K_n}|\overline{v}^{\varepsilon}(t,x)-v^{\varepsilon}(t,x)| = O(\varepsilon^{b+1}) \quad {\rm as}\ \varepsilon \downarrow 0, \\
& \sup_{(t,x) \in K_n}|\overline{w}^{\varepsilon}(t,x)-w^{\varepsilon}(t,x)| = O(\varepsilon^{b+1}) \quad {\rm as}\ \varepsilon \downarrow 0.
\end{align*}
Together with $(\ref{eq : inequality 1})$ and $(\ref{eq : inequality 2})$, these estimates imply that
\[
\sup_{(t,x) \in K} \left|\int_{-a-T}^{x} \dfrac{\overline{w}^{\varepsilon}(t,y) - \overline{v}^{\varepsilon}(t,y)}{2}\,dy\right| \ge \dfrac{1}{2}\varepsilon^{b}
\]
for sufficiently small $\varepsilon > 0$, that is, $(\overline{V},\overline{W})$ does not satisfy $(\ref{eq : property s})$. Hence $(\overline{V},\overline{W}) \notin \mathcal{M}([0,T] \times \R)$ and so $B((V,W);e^{-b-1}/2^{n+1}) \cap \mathcal{M}([0,T] \times \R) = \emptyset$. Thus $\mathcal{M}([0,T] \times \R)$ is a closed subspace of $\widetilde{\G}([0,T] \times \R)^2$.

For $(V,W) \in \mathcal{M}([0,T] \times \R)$, define two transformations $F_1(V,W)$, $F_2(V,W)$ by the right-hand sides of $(\ref{eq : integral equation})$, respectively, and put
\[
F(V,W) = (F_1(V,W),F_2(V,W)).
\]
The assertion of the theorem will hold if we show that $F$ is a contraction on $\mathcal{M}([0,T] \times \R)$.

In fact, if $F$ is a contraction on $\mathcal{M}([0,T] \times \R)$, then by the Banach fixed point theorem, $F$ has a fixed point $(V,W) \in \mathcal{M}([0,T] \times \R)$. We define
\[
U = \int_{-a-T}^{x} \dfrac{W-V}{2}\,dy \quad (\in \G([0,T] \times \R)).
\]
Noting that $(V,W)$ satisfies problem $(\ref{eq : system1})$, we find that $U$ is a solution of problem $(\ref{eq : generalized nonlinear wave equation})$.
That $U$ satisfies $(\ref{eq : singsupp of U})$ can be seen as follows.
Since $V$, $W$ satisfy property (I), so does $U$. Hence for any $K \Subset [0,T] \times \mathbb{R}$, the function $U$ is of bounded type on $K$.
The fact that $(V,W)$ satisfies problem $(\ref{eq : system1})$ gives the relationships
\[
D_-U = \dfrac{V}{\sqrt{2}}, \quad D_+U = \dfrac{W}{\sqrt{2}},
\]
and so for any $\alpha \in N$
\[
D_-^{\alpha}U = \dfrac{1}{\sqrt{2}}D_-^{\alpha-1} V, \quad D_+^{\alpha}U = \dfrac{1}{\sqrt{2}}D_+^{\alpha-1}W.
\]
It follows from the assumption that $V$, $W$ satisfy properties (II) and (III) that for any $K \Subset [0,T] \times \mathbb{R} \setminus (\Gamma_- \cup \Gamma_+)$ and $\alpha \in \mathbb{N}$, the functions $D_{\pm}^{\alpha} U$ are of bounded type on $K$. Since the first equation in problem $(\ref{eq : generalized nonlinear wave equation})$ can be rewritten as
\[
2D_+D_- U = E f(U),
\]
for any $K \Subset [0,T] \times \mathbb{R} \setminus (\Gamma_- \cup \Gamma_+)$ and $\alpha$, $\beta \in \mathbb{N}$, the function $D_-^{\alpha}D_+^{\beta} U$ is of bounded type on $K$. Thus we obtain that
\[
U \in \G^0([0,T] \times \R \setminus (\Gamma_- \cup \Gamma_+)),
\]
i.e., assertion $(\ref{eq : singsupp of U})$ follows.

We now prove that $F$ maps $\mathcal{M}([0,T] \times \R)$ into itself. Let $(V,W)$ be in $\mathcal{M}([0,T] \times \R)$. Then we can write
\begin{align}
F_1(V,W) & = V_0(x-t) + \int_{C_+(t,x)} g_{V,W}(\xi)\,d\xi, \label{eq : integral equation3}\\
F_2(V,W) & = W_0(x+t) + \int_{C_-(t,x)} g_{V,W}(\xi)\,d\xi.\label{eq : integral equation4}
\end{align}
We see from the definitions of $V_0$ and $W_0$ and property (C) that $V_0(x-t)$ and $W_0(x+t)$ belong to $\widetilde{\G}([0,T] \times \R)$.
We can apply Lemma $\ref{lemma : 2}$ to find that the integral
\[
\int_{-a-T}^{x} \dfrac{W-V}{2}\,dy
\]
belongs to $\widetilde{\G}([0,T] \times \R)$.
Together with assumption (B), this implies that $g_{V,W}$ is in $\widetilde{\G}([0,T] \times \R)$. Hence, by Lemma $\ref{lemma : 1}$, the two integrals in $(\ref{eq : integral equation3})$ and $(\ref{eq : integral equation4})$ are in $\widetilde{\G}([0,T] \times \R)$. Furthermore,
\begin{align*}
& \int_{-a-T}^{x} \dfrac{F_2(V,W) - F_1(V,W)}{2}\,dy \\
& \quad = \int_{-a-T}^{x} \dfrac{W_0(y+t) - V_0(y-t)}{2}\,dy \\
& \qquad + \int_{-a-T}^{x} \dfrac{\int_0^t g_{V,W}(\tau,y+t-\tau)\,d\tau - \int_0^t g_{V,W}(\tau,y-t+\tau)\,d\tau}{2}\,dy \\
& \quad = \int_{-a-T+t}^{x-t} \dfrac{W_0(y) - V_0(y)}{2}\,dy + \int_{x-t}^{x+t} \dfrac{W_0(y)}{2}\,dy - \int_{-a-T-t}^{-a-T+t} \dfrac{V_0(y)}{2}\,dy \\
& \qquad + \int_0^t \dfrac{\int_{x-t+\tau}^{x+t-\tau} g_{V,W}(\tau,z)\,dz - \int_{-a-T-t+\tau}^{-a-T+t-\tau} g_{V,W}(\tau,z)\,dz}{2}\,d\tau \\
& \quad =0
\end{align*}
on $\{(t,x) \in [0,T] \times \mathbb{R} : |x| > t + a\}$. Thus $F$ maps $\mathcal{M}([0,T] \times \R)$ into itself.

Finally we prove that $F$ is a contraction on $\mathcal{M}([0,T] \times \R)$. Let $(V_1,W_1)$, $(V_2,W_2)$ be in $\mathcal{M}([0,T] \times \R)$ and let $(v_1^{\varepsilon}, w_1^{\varepsilon})$, $(v_2^{\varepsilon}, w_2^{\varepsilon})$ be their representatives. We have to show that there is $\kappa < 1$ such that
\[
d(F(V_1,W_1),F(V_2,W_2)) \le \kappa d((V_1,W_1),(V_2,W_2)).
\]
By $(\ref{eq : integral equation3})$ and $(\ref{eq : integral equation4})$,
the mapping $F = (F_1,F_2)$ satisfies
\begin{align*}
F_1(v_1^{\varepsilon},w_1^{\varepsilon}) - F_1(v_2^{\varepsilon},w_2^{\varepsilon}) & = \int_{C_+(t,x)} (g_{v_1^{\varepsilon},w_1^{\varepsilon}}(\xi) - g_{v_2^{\varepsilon},w_2^{\varepsilon}}(\xi))\,d\xi,\\
F_2(v_1^{\varepsilon},w_1^{\varepsilon}) - F_2(v_2^{\varepsilon},w_2^{\varepsilon}) & = \int_{C_-(t,x)} (g_{v_1^{\varepsilon},w_1^{\varepsilon}}(\xi) - g_{v_2^{\varepsilon},w_2^{\varepsilon}}(\xi))\,d\xi.
\end{align*}
Similarly to the proof of Lemma \ref{lemma : 1}, the derivatives $D_{\pm}^{\alpha} (F_i(v_1^{\varepsilon},w_1^{\varepsilon}) - F_i(v_2^{\varepsilon},w_2^{\varepsilon}))$ can be calculated for $i = 1, 2$ and the following inequality holds: for any $m \in \N_0$, there exist $n \ge m$ and $C > 0$ such that
\begin{align*}
& \widetilde{\mu}_m( F_1(v_1^{\varepsilon},w_1^{\varepsilon}) - F_1(v_2^{\varepsilon},w_2^{\varepsilon})) +  \widetilde{\mu}_m( F_2(v_1^{\varepsilon},w_1^{\varepsilon}) - F_2(v_2^{\varepsilon},w_2^{\varepsilon})) \\
& \quad \le C \varepsilon \left(\widetilde{\mu}_n( v_1^{\varepsilon} - v_2^{\varepsilon}) + \widetilde{\mu}_n( w_1^{\varepsilon} - w_2^{\varepsilon})\right).
\end{align*}
From this estimate, one can pick a subsequence of the $\widetilde{\mu}_j$ such that (denoting the subsequence again by $\widetilde{\mu}_j)$
\begin{align*}
& \widetilde{\mu}_j(F_1(v_1^{\varepsilon},w_1^{\varepsilon}) - F_1(v_2^{\varepsilon},w_2^{\varepsilon})) + \widetilde{\mu}_j(F_2(v_1^{\varepsilon},w_1^{\varepsilon}) - F_2(v_2^{\varepsilon},w_2^{\varepsilon})) \\
& \quad \leq C \varepsilon \left(\widetilde{\mu}_{j+1}( v_1^{\varepsilon} - v_2^{\varepsilon}) + \widetilde{\mu}_{j+1}( w_1^{\varepsilon} - w_2^{\varepsilon})\right).
\end{align*}
This leads to
\begin{align*}
& \widetilde{p}_j(F_1(V_1,W_1) - F_1(V_2,W_2)) + \widetilde{p}_j(F_2(V_1,W_1) - F_2(V_2,W_2)) \\
& \quad \leq \exp(-\widetilde{\nu}_0(E)) \left(\widetilde{p}_{j+1}(V_1-V_2) + \widetilde{p}_{j+1}(W_1-W_2)\right).
\end{align*}
Note that any countable subfamily of $\{\widetilde{p}_n\}$ gives the same topology, since $\{\widetilde{p}_n\}$ is increasing. Hence
\[
d(F(V_1,W_1), F(V_2,W_2)) \leq 2\exp(-\widetilde{\nu}_0(E)) d((V_1,W_1),(V_2,W_2)),
\]
where the factor 2 comes from the shift $j\to j+1$. Since $\widetilde{\nu}_0(E) > \log 2$, it follows that $F$ is a contraction.
\qed
\end{pf}

\begin{remark}
(a) Condition (A) in Theorem \ref{thm : propagation1} can be weakened to the requirement that
$\widetilde{\nu}_0(E) > \log 2$.

(b) By the contraction mapping principle, the solution is unique in the space $\mathcal{M}([0,T] \times \R)$. As noted in the introduction, the solution is known to be unique in $\mathcal{G}([0,T] \times \mathbb{R})$ if $f$ is globally Lipschitz.
\end{remark}
\begin{remark}
In a similar way, one can show that if condition (C) in Theorem \ref{thm : propagation1} is replaced by
\begin{itemize}
\item[(C$^{\prime}$)] $U_0$, $U_1$ belong to $\mathcal{G}(\mathbb{R}) \cap \mathcal{G}^0(\mathbb{R}\setminus [-b,b])$ and are compactly supported. Furthermore, $U_0^{\prime}$, $U_1$ are of bounded type on $\R$,
\end{itemize}
then  for any $T > 0$, there exists a solution $U \in \mathcal{G}([0,T] \times \mathbb{R})$ to problem $(\ref{eq : generalized nonlinear wave equation})$ such that
\begin{eqnarray*}
\singsupp_{{\mathcal G^{0}}} U &\subset& \{(t,x) : t-b \le x \le t + b,\ 0 \le t \le T\}\\
   &&\qquad \cup\ \ \{(t,x) : -t-b \le x \le -t + b,\ 0 \le t \le T\}.
\end{eqnarray*}
\end{remark}
We complete this section by an example of a solution to a semilinear wave equation in three space dimension, in which the initial singularity does not propagate along the light cone, but rather contaminates the interior.
\begin{example}
Let
\begin{equation}\label{eq : repr}
   u_0^\eps(x) = \big(|x|^2 + \eps\big)^{-1/2}, \quad x\in\R^3.
\end{equation}
Clearly, $(u_0^\eps)_{\eps\in(0,1]}$ represents an element $U_0\in \mathcal{G}(\R^3) \cap \mathcal{G}^0(\R^3\setminus\{0\})$. Let $U\in\cG([0,\infty)\times\R^3)$ be represented by the same family \eqref{eq : repr}, i.e., $U$ does not depend on $t$. A simple calculation shows that $U$ solves the initial value problem
\begin{equation*}
\begin{array}{lr}
\partial_t^2U - \Delta U = 3EU^5 & \mbox{in}\ \G([0,\infty)\times\mathbb{R}^3),\vspace{4pt} \\
U|_{t = 0} = U_0,\quad \partial_tU|_{t = 0} = 0 & \mbox{in}\ \G(\mathbb{R}^3) \vspace{4pt} \\
\end{array}
\end{equation*}
and
\[
\singsupp_{{\mathcal G^{0}}} U = \{(t,x): x = 0, t\geq 0\}.
\]
\end{example}

\subsection*{Acknowledgments}
The results in this paper were obtained during several visits of the first author to Universit\"{a}t Innsbruck. He expresses his heartfelt thanks to the Unit of Engineering Mathematics for the warm hospitality during his visits.


\begin{thebibliography}{00}

\bibitem{Aragona:2012}
J. Aragona, R. Fernandez, S.O. Juriaans, M. Oberguggenberger.
Differential calculus and integration of generalized functions over membranes.
Monatsh. Math. {\bf 166} (2012), 1--18.

\bibitem{B:1990}
H. A. Biagioni, {\em A nonlinear theory of generalized functions}. Lect. Notes Math. 1421. Springer-Verlag, Berlin, 1990.

\bibitem{Beals:1989}
M. Beals. Propagation and interaction of singularities in nonlinear hyperbolic problems.
Progress in Nonlinear Differential Equations and their Applications, vol. 3. Birkh\"{a}user Boston,
Inc., Boston, MA 1989.

\bibitem{Colombeau:1984}
J.-F. Colombeau. \emph{New generalized functions and multiplication of distributions}.
North-Holland Mathematics Studies, vol. 84. North-Holland, Amsterdam 1984.

\bibitem{Colombeau:1985}
J.-F. Colombeau. {\it Elementary introduction to new generalized functions.}
North-Holland Mathematics Studies, vol. 113. North-Holland, Amsterdam 1985.


\bibitem{D:2011}
H. Deguchi. A linear first-order hyperbolic equation with a discontinuous
coefficient: distributional shadows and propagation of singularities.
Electron. J. Differential Equations  {\bf 2011} (2011), No. 76.

\bibitem{DHO:2013}
H. Deguchi, G. H\"{o}rmann, M. Oberguggenberger.  The wave equation with a discontinuous
coefficient depending on time only: generalized solutions and propagation of singularities.
In: S. Molahajloo, S. Pilipovi\'c, J. Toft, M. Wong (eds.), \emph{Pseudo-differential
operators, generalized functions and asymptotics}. Operator Theory: Advances and
Applications, vol. 231. Birkh\"{a}user, Basel 2013, 323--340.

\bibitem{DO:2016}
H. Deguchi, M. Oberguggenberger.
Propagation of singularities for generalized solutions to wave equations with discontinuous coefficients.
SIAM J. Math. Anal. {\bf 48} (2016), 397--442.

\bibitem{DMO:2008}
A. Delcroix, J.-A. Marti, M. Oberguggenberger. Microlocal asymptotic analysis in algebras of
generalized functions. Asymptotic Analysis {\bf 59} (2008), 83--107.


\bibitem{Garetto:2005}
C. Garetto, Topological structures in Colombeau algebras: topological $\widetilde{\mathbb C}$-modules and duality theory.
 Acta Appl. Math.  {\bf 88}  (2005), 81--123.

\bibitem{Garetto:2006}
C. Garetto. Microlocal analysis in the dual of a Colombeau algebra:
generalized wave front sets and noncharacteristic regularity.
New York J. Math. {\bf 12}  (2006), 275--318.

\bibitem{GarettoGO:2005}
C. Garetto, T. Gramchev, M. Oberguggenberger. Pseudodifferential operators with generalized symbols and
regularity theory. Electron. J. Diff. Eqns. {\bf 2005} (2005), No. 116.

\bibitem{GarettoHoer:2006}
C. Garetto, G. H\"{o}rmann.  On duality theory and pseudodifferential
techniques for Colombeau algebras: generalized delta functionals, kernels and wave front sets.
Bull. Cl. Sci. Math. Nat. Sci. Math.  {\bf 31}  (2006), 115--136.


\bibitem{GarettoHO:2009}
C. Garetto, G. H\"{o}rmann. Generalized Oscillatory Integrals and Fourier Integral
Operators. Proc. Edinburgh Math. Soc. {\bf 52} (2009), 351--386.

\bibitem{GarettoO:2014}
C. Garetto, M. Oberguggenberger. Generalized Fourier integral
operator methods for hyperbolic equations with singularities.
Proc. Edinb. Math. Soc. {\bf  57}  (2014), 423--463.


\bibitem{GKOS:2001}
M. Grosser, M. Kunzinger, M. Oberguggenberger, R. Steinbauer, {\em Geometric theory of generalized functions with applications to general relativity}. Mathematics and its Applications, vol. 537. Kluwer Acad. Publ., Dordrecht, 2001.


\bibitem{HdeH:2001}
G. H\"{o}rmann, M.V. de Hoop.  Microlocal analysis and global solutions of
some hyperbolic equations with discontinuous coefficients.
Acta Appl. Math. {\bf 67}  (2001), 173--224.


\bibitem{HOP:2006}
G. H\"{o}rmann, M. Oberguggenberger, S. Pilipovi\'{c}. Microlocal hypoellipticity of linear partial
differential operators with generalized functions as coefficients.
Trans. Am. Math. Soc. {\bf 358} (2006), 3363--3383.


\bibitem{NOP:2005}
M. Nedeljkov, M. Oberguggenberger, S. Pilipovi\'{c}.
Generalized solutions to a semilinear wave equation. Nonlinear Anal. {\bf 61} (2005), 461--475.


\bibitem{NPS:1998}
M.~Nedeljkov, S.~Pilipovi{\'c},  D.~Scarpal{\'e}zos,
\newblock {\em The linear theory of {C}olombeau generalized functions}. Pitman Research Notes Math., vol. 385.
\newblock Longman, Harlow 1998.

\bibitem{O:1987}
M. Oberguggenberger. Generalized solutions to semilinear hyperbolic systems.
Monatsh. Math. {\bf 103}  (1987), 133--144.

\bibitem{O:1992}
M. Oberguggenberger, {\em Multiplication of distributions and applications to partial differential equations}. Pitman Research Notes Math., vol. 259. Longman Scientific {\&} Technical, Harlow 1992.

\bibitem{O:2004}
M. Oberguggenberger,
Generalized solutions to nonlinear wave equations. Matem\'{a}tica Contempor\^{a}nea {\bf 27} (2004), 169--187.

\bibitem{O:2006}
M. Oberguggenberger, Regularity theory in Colombeau algebras.  Bull. T. CXXXIII Acad. Serbe Sci. Arts, Cl. Sci. Math. Nat., Sci. Math. {\bf 31} (2006), 147--162.

\bibitem{O:2006a}
M. Oberguggenberger, Notes on regularity results in Colombeau algebras. Integral Transforms and Special Functions {\bf 17} (2006), 101--107.

\bibitem{O:2008}
M. Oberguggenberger. Hyperbolic systems with discontinuous coefficients: generalized
wavefront sets. In: L. Rodino, M.W. Wong (eds.), \emph{New developments in pseudo-differential
operators}. Operator Theory: Advances and Applications, vol. 189. Birkh\"{a}user, Basel 2008, 117--136.

\bibitem{MORusso:1998}
M. Oberguggenberger, F. Russo.  Nonlinear stochastic wave equations.
Integral Transform. Spec. Funct. {\bf 6}(1-4) (1998), 71--83.

\bibitem{Rauch:1979}
J. Rauch. Singularities of solutions to semilinear wave equations.
J. Math. Pures Appl. {\bf 58} (1979), 299--308.


\bibitem{RauchReed:1981}
J. Rauch, M. Reed. Jump discontinuities of semilinear, strictly hyperbolic systems in two
variables: creation and propagation.
Comm. Math. Phys. {\bf 81}  (1981), 203--227.


\bibitem{Reed:1978}
M. C. Reed, Propagation of singularities for non-linear wave equations in one dimension.
Comm. Partial Differential Equations {\bf 3} (1978), 153--199.

\bibitem{Scarpalezos:1998}
D. Scarpal\'{e}zos, Some remarks on functoriality of Colombeau's construction; topological and microlocal aspects and applications.
Integral Transform. Spec. Funct.  {\bf 6}  (1998), 295--307.

\bibitem{Scarpalezos:2000}
D. Scarpal\'{e}zos, Colombeau's generalized functions: topological structures; microlocal properties. A simplified point of view. I.
Bull. Cl. Sci. Math. Nat. Sci. Math.  {\bf 25}  (2000), 89--114.

\bibitem{Travers:1997}
K.E. Travers. Semilinear hyperbolic systems in one space dimension with
strongly singular initial data.
Electron. J. Differential Equations  {\bf 1997} (1997), No. 14.

\end{thebibliography}
\end{document}